\documentclass[11pt,leqno]{article}
\usepackage[margin=1in]{geometry} 
\usepackage{amssymb,amsfonts,amsmath,bbm,mathrsfs,stmaryrd,mathtools}
\usepackage{xcolor}
\usepackage{url}

\usepackage{extarrows}

\usepackage[shortlabels]{enumitem}
\usepackage{tensor}

\usepackage{xr}
\usepackage[T1]{fontenc}
\usepackage[utf8]{inputenc}
\usepackage[colorlinks,
linkcolor=black!75!red,
citecolor=blue,
pdftitle={},
pdfproducer={pdfLaTeX},
pdfpagemode=None,
bookmarksopen=true,
bookmarksnumbered=true,
backref=page]{hyperref}

\usepackage{tikz}
\usetikzlibrary{arrows,calc,decorations.pathreplacing,decorations.markings,decorations.shapes,intersections,shapes.geometric,through,fit,shapes.symbols,positioning,decorations.pathmorphing}

\makeatletter
\newcommand*\circled[1]{\tikz[baseline=(char.base)]{
    \node[shape=circle, draw, inner sep=0pt, 
    minimum height={\f@size},] (char) {\vphantom{WAH1g}#1};}}
\makeatother

\newcommand{\operp}{\text{\circled{$\perp$}}}

\makeatletter
\DeclareRobustCommand\widecheck[1]{{\mathpalette\@widecheck{#1}}}
\def\@widecheck#1#2{%
    \setbox\z@\hbox{\m@th$#1#2$}%
    \setbox\tw@\hbox{\m@th$#1%
       \widehat{%
          \vrule\@width\z@\@height\ht\z@
          \vrule\@height\z@\@width\wd\z@}$}%
    \dp\tw@-\ht\z@
    \@tempdima\ht\z@ \advance\@tempdima2\ht\tw@ \divide\@tempdima\thr@@
    \setbox\tw@\hbox{%
       \raise\@tempdima\hbox{\scalebox{1}[-1]{\lower\@tempdima\box
\tw@}}}%
    {\ooalign{\box\tw@ \cr \box\z@}}}
\makeatother

\usepackage{braket}

\usepackage[amsmath,thmmarks,hyperref]{ntheorem}
\usepackage{cleveref}

\newcommand\nthalias[1]{\AddToHook{env/#1/begin}{\crefalias{lemma}{#1}}}

\nthalias{definition}
\nthalias{example}
\nthalias{examples}
\nthalias{remark}
\nthalias{remarks}
\nthalias{convention}
\nthalias{notation}
\nthalias{construction}
\nthalias{sketch}
\nthalias{theoremN}
\nthalias{propositionN}
\nthalias{corollaryN}
\nthalias{lemma}
\nthalias{proposition}
\nthalias{corollary}
\nthalias{theorem}
\nthalias{conjecture}
\nthalias{question}
\nthalias{assumption}

\creflabelformat{enumi}{#2#1#3}

\crefname{section}{Section}{Sections}
\crefformat{section}{#2Section~#1#3} 
\Crefformat{section}{#2Section~#1#3} 

\crefname{subsection}{\S}{\S\S}
\AtBeginDocument{%
  \crefformat{subsection}{#2\S#1#3}%
  \Crefformat{subsection}{#2\S#1#3}% 
}

\crefname{subsubsection}{\S}{\S\S}
\AtBeginDocument{%
  \crefformat{subsubsection}{#2\S#1#3}%
  \Crefformat{subsubsection}{#2\S#1#3}% 
}

\theoremstyle{plain}

\newtheorem{lemma}{Lemma}[section]
\newtheorem{proposition}[lemma]{Proposition}
\newtheorem{corollary}[lemma]{Corollary}
\newtheorem{theorem}[lemma]{Theorem}

\theoremstyle{plain}
\theoremnumbering{Alph}

\theoremstyle{plain}
\theorembodyfont{\upshape}
\theoremsymbol{\ensuremath{\blacklozenge}}

\newtheorem{definition}[lemma]{Definition}

\newtheorem{remark}[lemma]{Remark}
\newtheorem{remarks}[lemma]{Remarks}

\newtheorem{notation}[lemma]{Notation}

\crefname{definition}{definition}{definitions}
\crefformat{definition}{#2definition~#1#3} 
\Crefformat{definition}{#2Definition~#1#3} 

\crefname{ex}{example}{examples}
\crefformat{example}{#2example~#1#3} 
\Crefformat{example}{#2Example~#1#3} 

\crefname{exs}{example}{examples}
\crefformat{examples}{#2example~#1#3} 
\Crefformat{examples}{#2Example~#1#3} 

\crefname{remark}{remark}{remarks}
\crefformat{remark}{#2remark~#1#3} 
\Crefformat{remark}{#2Remark~#1#3} 

\crefname{remarks}{remark}{remarks}
\crefformat{remarks}{#2remark~#1#3} 
\Crefformat{remarks}{#2Remark~#1#3} 

\crefname{convention}{convention}{conventions}
\crefformat{convention}{#2convention~#1#3} 
\Crefformat{convention}{#2Convention~#1#3} 

\crefname{notation}{notation}{notations}
\crefformat{notation}{#2notation~#1#3} 
\Crefformat{notation}{#2Notation~#1#3} 

\crefname{table}{table}{tables}
\crefformat{table}{#2table~#1#3} 
\Crefformat{table}{#2Table~#1#3}

\crefname{lemma}{lemma}{lemmas}
\crefformat{lemma}{#2lemma~#1#3} 
\Crefformat{lemma}{#2Lemma~#1#3} 

\crefname{proposition}{proposition}{propositions}
\crefformat{proposition}{#2proposition~#1#3} 
\Crefformat{proposition}{#2Proposition~#1#3} 

\crefname{propositionN}{proposition}{propositions}
\crefformat{propositionN}{#2proposition~#1#3} 
\Crefformat{propositionN}{#2Proposition~#1#3} 

\crefname{corollary}{corollary}{corollaries}
\crefformat{corollary}{#2corollary~#1#3} 
\Crefformat{corollary}{#2Corollary~#1#3} 

\crefname{corollaryN}{corollary}{corollaries}
\crefformat{corollaryN}{#2corollary~#1#3} 
\Crefformat{corollaryN}{#2Corollary~#1#3} 

\crefname{theorem}{theorem}{theorems}
\crefformat{theorem}{#2theorem~#1#3} 
\Crefformat{theorem}{#2Theorem~#1#3} 

\crefname{theoremN}{theorem}{theorems}
\crefformat{theoremN}{#2theorem~#1#3} 
\Crefformat{theoremN}{#2Theorem~#1#3} 

\crefname{enumi}{}{}
\crefformat{enumi}{#2#1#3}
\Crefformat{enumi}{#2#1#3}

\crefname{assumption}{assumption}{Assumptions}
\crefformat{assumption}{#2assumption~#1#3} 
\Crefformat{assumption}{#2Assumption~#1#3} 

\crefname{construction}{construction}{Constructions}
\crefformat{construction}{#2construction~#1#3} 
\Crefformat{construction}{#2Construction~#1#3} 

\crefname{sketch}{sketch}{Sketches}
\crefformat{sketch}{#2sketch~#1#3} 
\Crefformat{sketch}{#2Sketch~#1#3} 

\crefname{question}{question}{Questions}
\crefformat{question}{#2question~#1#3} 
\Crefformat{question}{#2Question~#1#3} 

\crefname{equation}{}{}
\crefformat{equation}{(#2#1#3)} 
\Crefformat{equation}{(#2#1#3)}

\numberwithin{equation}{section}

\theoremstyle{nonumberplain}
\theoremsymbol{\ensuremath{\blacksquare}}

\newtheorem{proof}{Proof}
\newcommand\pf[1]{\newtheorem{#1}{Proof of \Cref{#1}}}

\newcommand\bC{{\mathbb C}}

\newcommand\bF{{\mathbb F}}
\newcommand\bG{{\mathbb G}}

\newcommand\bP{{\mathbb P}}

\newcommand\bR{{\mathbb R}}
\newcommand\bS{{\mathbb S}}

\newcommand\bZ{{\mathbb Z}}

\newcommand\cC{{\mathcal C}}

\newcommand\cH{{\mathcal H}}

\newcommand\cK{{\mathcal K}}

\newcommand\cM{{\mathcal M}}

\newcommand\cV{{\mathcal V}}
\newcommand\cW{{\mathcal W}}

\newcommand\wh{\widehat}

\DeclareMathOperator{\Ad}{Ad}
\DeclareMathOperator{\id}{id}

\DeclareMathOperator{\GL}{GL}

\DeclareMathOperator{\SU}{SU}

\DeclareMathOperator{\G}{G}
\DeclareMathOperator{\orth}{O}

\DeclareMathOperator{\U}{U}

\newcommand{\qedhere}{\mbox{}\hfill\ensuremath{\blacksquare}}
\newcommand{\crc}[1]{\overset{\circ}{#1}}

\title{The combinatorics of permuting and preserving curve-bound spectra}
\author{Alexandru Chirvasitu}

\begin{document}

\date{}

\newcommand{\Addresses}{{% additional braces for segregating \footnotesize
  \bigskip
  \footnotesize

  \textsc{Department of Mathematics, University at Buffalo}
  \par\nopagebreak
  \textsc{Buffalo, NY 14260-2900, USA}  
  \par\nopagebreak
  \textit{E-mail address}: \texttt{achirvas@buffalo.edu}

}}

\maketitle

\begin{abstract}
  We prove that continuous spectrum- and commutativity-preserving maps to $\mathcal{M}_n(\mathbb{C})$ from the space of normal (real or complex) $n\times n$, $n\ge 3$ matrices with spectra contained in a given continuous-injection interval image $\Lambda\subseteq \mathbb{C}$ or $\mathbb{R}$ are (a) conjugations; (b) transpose conjugations, or (c) orderings of spectra according to an orientation of $\Lambda$, with fixed eigenspaces. This generalizes results of Petek's (self-maps of real or complex Hermitian matrices) and the author's (complex Hermitian matrices as the domain, $\mathcal{M}_n(\mathbb{C})$ as the codomain). An application rules out possibility (c) for normal matrices with spectra constrained to a simple closed curve, extending a result by the author, Gogi\'c and Toma\v{s}evi\'c to the effect that continuous commutativity and spectrum preservers on unitary groups are (transpose) conjugations.

  The involution preserving eigenspaces and complex-conjugating eigenvalues is a novel possibility beyond (a), (b) and (c) if the domain consists of all semisimple operators with $\Lambda$-bound spectra instead; its continuity (or lack thereof) and whether or not that map furthermore extends continuously to arbitrary $\Lambda$-constrained-spectrum matrices hinge on the geometry and regularity of $\Lambda$.
\end{abstract}


\vspace{.5cm}

\noindent{MSC 2020: 47A10; 15B57; 54D05; 15A27; 46C05; 20B30; 54H15; 54F50
  
  % 47A10 Spectrum, resolvent
  % 15B57 Hermitian, skew-Hermitian, and related matrices
  % 54D05 Connected and locally connected spaces (general aspects)
  % 15A27 Commutativity of matrices
  % 46C05 Hilbert and pre-Hilbert spaces: geometry and topology (including spaces with semidefinite inner product)
  % 20B30 Symmetric groups
  % 54H15 Transformation groups and semigroups (topological aspects)
  % 54F50 Topological spaces of dimension $\leq 1$; curves, dendrites
}

%\tableofcontents

%%%%%%%%%%%%%%%%%%%%%%%%%%%%%%%%
%%%%%%%%%%%%%%%%%%%%%%%%%%%%%%%%
\section*{Introduction}

The problem of whether a linear spectrum-preserving map of a complex semisimple Banach algebra onto another is necessarily a \emph{Jordan morphism} \cite[\S 2.1.2]{ho-s_jord} was posed in \cite[\S 1]{Aupetit}, as a variant of a discussion in \cite[\S 9]{Kaplansky}; it has since spanned a considerable body of work on characterizing maps between various types of operator algebras (much of it in the context of matrix algebras) subject to various spectrum and algebraic-structure preservation constraints: \cite{MR2430550,MR1423038,MR1653251,MR2115007,MR4927632,GogicPetekTomasevic,MR4830482,Petek-TM,zbMATH01100760,PetekSemrl} and their own references will provide an ample overview. 

One specific result we will revisit below is \cite[Main theorem]{Petek-HM}, classifying spectrum- and commutativity-preserving continuous self-maps of the space $\cH_n(\Bbbk)$ of Hermitian $n\times n$ matrices over $\Bbbk\in \left\{\bR,\bC\right\}$: they are precisely
\begin{itemize}
\item the conjugations (necessarily by unitary/orthogonal operators);
\item the transpose conjugations;

\item or, up to conjugation,
  \begin{equation*}
    X
    \xmapsto{\quad}
    \mathrm{diag}\left(\lambda_1(X)\le \cdots\le \lambda_n(X)\right)
    \quad
    \left(\text{ordered spectrum of $X$}\right).
  \end{equation*}
\end{itemize}
The complex version is strengthened slightly in \cite[Theorem B]{2505.19393v3} by extending the codomain to $\cM_n(\bC)$ as a consequence of an analogue \cite[Theorem A]{2505.19393v3} for commutativity/spectrum preservers defined on the special unitary group $\SU(n)$ and combinatorial considerations pertaining to symmetric (or more generally, \emph{Coxeter} \cite[\S 5.1]{hmph_cox}) groups. 

The present paper proposes to distill the combinatorial core common to the results just mentioned, formalizing the conditions affording the type of spectrum-ordering-based arguments in evidence therein. The remainder of the current section elaborates. 

We will be working with subsets of $\cM_n(\Bbbk)$ ($n\times n$ matrices over a field $\Bbbk\in \left\{\bR,\bC\right\}$; occasionally plain $\cM_n$) generalizing the spaces of Hermitian matrices in having their spectra constrained to \emph{simple curves} in the ambient $\Bbbk$ (connected, unless specified otherwise). For our purposes, the phrase refers to the image of a continuous injection $I\lhook\joinrel\xrightarrow{\gamma}\Bbbk$ defined on an interval $I\subseteq \bR$ (closed or open or half-open, bounded or not). In this context we occasionally conflate maps $\gamma$ and their images; for $\Bbbk=\bR$, naturally, simple curves are nothing but intervals. 

\begin{notation}\label{not:hlgamma}
  \begin{enumerate}[(1),wide]
  \item Fix a subset $\Lambda\subseteq \Bbbk\in \left\{\bR,\bC\right\}$, an $n\in \bZ_{\ge 1}$.

    We write $\cH_{n\mid \Lambda}=\cH_{n\mid \Lambda}(\Bbbk)\subseteq \cM_n(\Bbbk)$ for any one of the following spaces of $n\times n$ matrices:
    \begin{itemize}[wide]
    \item having spectra contained in $\Lambda$;
    \item and possibly also \emph{semisimple} (i.e. diagonalizable over $\bC$) or \emph{normal} (i.e. commuting with their adjoints).
    \end{itemize}
    The symbol thus stands for any one of several spaces. In practice the distinction will not matter much (which is why it is convenient to have common notation); when not relying solely on context to distinguish we use superscripts: $\cH^{\circ}$, $\cH^{ss}$ and $\cH^*$ for arbitrary, semisimple and normal respectively. Alternatively, the superscript in $\cH^{\bullet}$ might serve as a collective placeholder. 
    
  \item We extend the notation to
    \begin{equation*}
      \cH_{n\mid \gamma}
      :=
      \cH_{n\mid \gamma(I)}
      ,\quad
      I\xrightarrow[\quad\text{map}\quad]{\quad\gamma\quad}\Bbbk.
    \end{equation*}
    In particular, this applies to simple curves. 
  \end{enumerate}
\end{notation}

\begin{theorem}\label{th:gen.herm}
  For $n\in \bZ_{\ge 3}$ and a simple curve $I\lhook\joinrel\xrightarrow{\iota}\Bbbk\in \left\{\bR,\bC\right\}$ the continuous, commutativity- and spectrum-preserving maps $\cH^*_{n\mid \gamma}(\Bbbk)\xrightarrow{\phi}\cM_n(\bC)$ are precisely those of one of the following two types.
  \begin{enumerate}[(a),wide]
  \item\label{item:th:gen.herm:cj} conjugation $\Ad_T:=T\left(-\right)T^{-1}$ or transpose conjugation $\Ad_T\circ (-)^t$ for some $T\in GL(n,\bC)$;

  \item\label{item:th:gen.herm:fix} or of the form
    \begin{equation*}
      \cH^*_{n\mid \gamma}
      \ni
      X
      \xmapsto{\quad\phi\quad}
      \Ad_T\mathrm{diag}\left(\lambda_1(X),\ \cdots,\ \lambda_n(X)\right)
    \end{equation*}
    where $\lambda_i(X)$ constitute the spectrum of $X$, arranged so that $\left(\gamma^{-1}\lambda_i(X)\right)_i\subset I$ is non-decreasing in $i$.
  \end{enumerate}
\end{theorem}

This generalizes a number of results in the literature in a few ways. 
\begin{itemize}[wide]
\item The case $\cH^{*}_{n\mid \bR\subset \bC}$ is that of ordinary Hermitian matrices (hence \Cref{se:curv}'s title below), recovering \cite[Theorem B]{2505.19393v3}. In turn, that result slightly expanded the complex-Hermitian half of \cite[Main theorem]{Petek-HM} by allowing $M_n(\bC)$ (as opposed to only $\cH^{*}_{n\mid \bR\subset \bC}$) as a codomain.

\item The real case $\cH^*_{n\mid \bR}(\bR)$ generalizes the real-Hermitian half of the same \cite[Main theorem]{Petek-HM}, again by enlarging the codomain to all of $M_n(\bC)$.
\end{itemize}

\Cref{th:gen.herm} can also, incidentally, serve as a precursor for a generalized-\emph{unitary} version. As a particular case, one can recover the unitary version of \cite[Theorem 2.1]{2501.06840v2}, to the effect that continuous commutativity-and-spectrum preservers $\U(n)\to \cM_n(\bC)$ are type-\Cref{item:th:gen.herm:cj}: take $\Lambda:=\bS^1\subset \bC$ in \Cref{th:scl} below. 

Simple \emph{closed} (as opposed to plain) curves in topological spaces are subspaces thereof homeomorphic to the circle $\bS^1$ (matching the terminology of \cite[pre Theorem 61.3]{mnk}, for example).

\begin{theorem}\label{th:scl}
    For $n\in \bZ_{\ge 3}$ and a simple closed curve $\Lambda\subset \bC$ the continuous, commutativity- and spectrum-preserving maps $\cH^*_{n\mid \Lambda}(\Bbbk)\xrightarrow{\phi}\cM_n(\bC)$ are precisely the conjugations or transpose conjugations. 
\end{theorem}

The situation differs drastically for $\cH^{\circ,ss}$: there are somewhat surprising new possibilities for what a continuous commutativity and spectrum preserver might look like; there is also an added (and perhaps surprising) caveat that the geometric/analytic regularity properties of $\Lambda$ seem to play a role in whether or not such candidates fulfill the requirements (the sticking point being continuity). A few reminders will prepare the terrain.

\begin{itemize}[wide]
\item The \emph{positive} operators \cite[Definition 5.2.1]{li_real-oa_2003} (here, matrices) are the normal ones with non-negative spectrum; a `$\le 0$' subscript will indicate positivity in various matrix spaces.

\item Arbitrary $T\in \GL(\Bbbk)$ admit a \emph{polar decomposition} (\cite[\S I.5.2.2]{blk}, going through as expected \cite[Theorem 1.2.5]{li_real-oa_2003} for \emph{real} $C^*$-algebras also)
  \begin{equation*}
    T=|T^*|U
    ,\quad
    |T^*|:=\left(TT^*\right)^{1/2}
    ,\quad
    U\text{ unitary/orthogonal}.
  \end{equation*}

\item Consequently, as isometric conjugations preserve normality, semisimple operators are positive conjugates of normal ones; this will be implicit in the statement of \Cref{th:if.l.reg}.
\end{itemize}

There is also the matter of ``regularity'' for $\Lambda$, referenced above; this refers to variants of $C^k$-differentiability for arbitrary $\Lambda\subset \bC$. Write $\cC^k \Lambda$ for the \emph{configuration space} \cite[Definition 1.1]{zbMATH05785888} of distinct $k$-tuples of a set $E$. For a symmetric map $\cC^n\Lambda\xrightarrow{f}\bC$, $\Lambda\subseteq \bC$ denote by
\begin{equation*}
  \cC^{n+1}\Lambda
  \ni
  (x_0\cdots x_n)
  \xmapsto{\quad\Delta f\quad}
  \frac{f(x_1\cdots x_n)-f(x_0\cdots x_{n-1})}{x_n-x_0}
\end{equation*}
its image through the \emph{difference-quotient operator} $\Delta$. Following \cite[Definition 3.1]{zbMATH06285212} (with an index shift), a function $\Lambda\xrightarrow{f}\bC$ is
\begin{itemize}[wide]
\item \emph{DB$^n$} if $\Delta^{n-1} f$ (an $n$-variable function) is bounded locally around any diagonal point $(xx\cdots x)$ for \emph{cluster points} \cite[Definition 4.9]{wil_top} $x\in \Lambda$;

\item and \emph{DC$^n$} if $\Delta^{n-1} f$ has a finite limit at every $(xx\cdots x)$ for cluster points $x\in \Lambda$. 
\end{itemize}

\begin{theorem}\label{th:if.l.reg}
  Let $n\in \bZ_{\ge 2}$ and $\Lambda\subseteq \Bbbk\in \left\{\bR,\bC\right\}$ a simple curve.
  \begin{enumerate}[(1),wide]

  \item\label{item:th:if.l.reg:is.cs} All compositions of
    \begin{itemize}[wide]
    \item type-\Cref{item:th:gen.herm:cj} or -\Cref{item:th:gen.herm:fix} maps from \Cref{th:gen.herm};
    \item and the involution
      \begin{equation}\label{eq:rev.map}
        \cH^{ss}_{n\mid\Lambda}
        \ni
        \Ad_R N
        =
        X
        \xmapsto{\quad (-)^{\rho}\quad}
        X^{\rho}
        :=
        \Ad_{R^{-1}}N
        ,\quad
        \left[
          \begin{gathered}
            N\in \cH^*_{n\mid \Lambda}\\
            R\in \GL(\Bbbk)_{\ge 0}
          \end{gathered}
        \right.
      \end{equation}
    \end{itemize}
    are commutativity and spectrum preservers $\cH^{ss}_{n\mid\Lambda}\to \cM_n(\bC)$. 
    
  \item\label{item:th:if.l.reg:ddb} Said maps are all continuous if and only if the complex conjugation map $\overline{(-)}$ is DB$^{n}$ on $\Lambda$.
    
  \item\label{item:th:if.l.reg:ddc} Furthermore, the maps from \Cref{item:th:if.l.reg:ddb} above all extend continuously to $\cH^{\circ}_{n\mid \Lambda}$ if and only if $\overline{(-)}|_{\Lambda}$ is DC$^{n}$.
  \end{enumerate}
\end{theorem}

Cf. \cite[Proposition 2.8]{2501.06840v2} for another appearance, in a parallel context, of the selfsame map \Cref{eq:rev.map}. The extent to which \Cref{th:if.l.reg} admits a converse and the possible shape such a converse might take are the subject of future work.

%%%%%%%%%%%%%%%%%%%%%%%%%%%%%%%%
%%%%%%%%%%%%%%%%%%%%%%%%%%%%%%%%
\section{Curve-constrained generalized Hermitian matrices}\label{se:curv}

We refer to maps preserving commutativity (or spectra, or both) as \emph{C, S or CS preservers} respectively for brevity (the respective adjective phrases would be \emph{CS-preserving}, etc.).

\begin{remark}\label{re:+not}
  Observe that maps of type either \Cref{item:th:gen.herm:cj} or \Cref{item:th:gen.herm:fix} certainly do meet the requirements, so the focus throughout will be on the converse. It will also be convenient to indicate spaces of \emph{simple} operators (i.e. those with simple spectrum) contained in the various $\cH$ with an additional `$+$' superscript: $\cH_n^{ss+}(\bR)$, $\cH_n^{*+}$, etc.
\end{remark}

Some terminology will help streamline some of the discussion. 

\begin{definition}\label{def:orient}
  A \emph{parametrization} of a simple curve $\Lambda\subseteq X$ in a topological space $X$ is a continuous bijection $I\xrightarrow{\gamma}\Lambda$ for some interval $I\subseteq \bR$. 
  
  An \emph{orientation} of $\Lambda$ is a class of parametrizations, two declared equivalent whenever they fit into a commutative triangle
  \begin{equation*}
    \begin{tikzpicture}[>=stealth,auto,baseline=(current  bounding  box.center)]
      \path[anchor=base] 
      (0,0) node (l) {$I$}
      +(2,.5) node (u) {$I'$}
      +(4,0) node (r) {$\Lambda$}
      ;
      \draw[->] (l) to[bend left=16] node[pos=.5,auto,swap] {$\scriptstyle \cong$} node[pos=.5,auto] {$\scriptstyle \text{increasing homeomorphism}$} (u);
      \draw[->] (u) to[bend left=6] node[pos=.5,auto] {$\scriptstyle \gamma'$} (r);
      \draw[->] (l) to[bend right=6] node[pos=.5,auto,swap] {$\scriptstyle \gamma$} (r);
    \end{tikzpicture}    
  \end{equation*}
  A parametrization in the class singled out by the orientation is \emph{compatible} with that orientation, or simply \emph{positive} (with respect to the orientation). 
  
  For a simple curve $\Lambda$ with a fixed orientation (i.e. an \emph{oriented} simple curve) we write $\lambda\le \lambda'\in \Lambda$ if $\gamma^{-1}\lambda\le \gamma^{-1}\lambda'$ for some orientation-positive parametrization $I\xrightarrow{\gamma}\Lambda$.
\end{definition}

\begin{notation}\label{not:flg}
  \begin{enumerate}[(1),wide]
  \item For $\Bbbk\in \left\{\bR,\bC\right\}$ and a finite-dimensional $\Bbbk$-Hilbert space (mostly the standard $\Bbbk^n$, $n\in \bZ_{\ge 1}$)
    \begin{equation*}
      \bF^{\bullet}(V)
      :=
      \left\{
        \text{$\left(\dim V\right)$-tuples of lines in $V$, }
        \left[
          \begin{aligned}
            \text{linearly independent}&\text{ if $\bullet=ss$}\\
            \text{mutually orthogonal}&\text{ if $\bullet=*$}
          \end{aligned}
        \right.
      \right\}
      \subset
      \left(\bP V\right)^{\dim V}
    \end{equation*}
    (variants of the usual \cite[\S 2]{MR106911} \emph{flag variety} attached to $V$). 

    Note that all $\bF^{\bullet}$ are equipped with free (left) actions
    \begin{equation*}
      \bF^{\bullet}(\Bbbk)
      \ni
      (\ell_i)_{i=1}^{\dim V}
      =:
      x
      \xmapsto{\quad\theta\in S_{\dim V}\quad}
      \theta x
      :=
      (\ell_{\theta^{-1} i})_{i=1}^{\dim V}
      \in
      \bF^{\bullet}(\Bbbk)
    \end{equation*}
    by the respective symmetric groups $S_{\dim V}$, permuting the lines of each independent/orthogonal tuple.

  \item More generally, consider a partition
    \begin{equation*}
      \mu=\left(\mu_1\ge \cdots\ge \mu_s>0\right)
      ,\quad
      \sum_j\mu_j=n
      \quad
      \left(\text{shorthand: }\mu\vdash n\right).
    \end{equation*}
    We conflate $\mu$ and the associated\emph{Young diagram} \cite[Notation]{fult_y_1997}: left-aligned rows of boxes of respective lengths $\lambda_i$, longest rows placed higher up.
    
    Set
    \begin{equation*}
      \bF_{\mu}^{\bullet}(V)
      :=
      \left\{
        \left(V_1,\ \cdots,\ V_s\right)
        \ :\
        V_j\le V
        ,\quad
        \dim V_j=\mu_j
        ,\quad
        \left[
        \begin{aligned}
          V_j\perp V_{j'\ne j}&\text{ if $\bullet=ss$}\\
          \sum V_j=V&\text{ if $\bullet=*$}
        \end{aligned}
        \right.
      \right\}.
    \end{equation*}
    This recovers the previous construct as $F^{\bullet}=F^{\bullet}_{(11\cdots 1)}$. 
  \end{enumerate}
\end{notation}

\begin{remarks}\label{res:fss}
  \begin{enumerate}[(1),wide]

  \item\label{item:res:fss:fss} In its $\bullet=ss$ variant \Cref{not:flg} makes sense over arbitrary fields, as Hilbert-space structures play no role.
    
  \item\label{item:res:fss:fconn} The spaces $\bF^{\bullet}(\Bbbk)$, $\Bbbk\in \left\{\bR,\bC\right\}$ can be identified with quotients $\G_n/\Bbbk^n$ by the actions scaling the columns of $\G_n$, where $\G\in \left\{\GL(\Bbbk),\orth,\U\right\}$ (general linear, orthogonal, unitary). In particular, said spaces are all connected.
  \end{enumerate}
\end{remarks}

Spectra in $\cH_{n\mid \Lambda}$ being orderable for an oriented curve $\Lambda\subseteq \Bbbk$, we have maps
\begin{equation}\label{eq:h2f}
  \cH^{\bullet +}_{n\mid \Lambda}(\Bbbk)
  \ni
  T
  \xmapsto[\quad\lambda_1\le \cdots\le \lambda_n\quad]{\quad\quad}
  \left(\text{$\lambda_i(T)$-eigenspace}\right)_{i=1}^n
  \in
  \bF^{\bullet}(\Bbbk^n)
\end{equation}
(continuous, by a variant of \cite[Proposition 13.4]{salt_divalg}, say). We retain below the convention adopted in \Cref{th:gen.herm} (and in place also in \Cref{eq:h2f}) of writing $\left(\lambda_i(T)\right)_{i=1}^n$ for the spectrum of $T\in \cH_{n\mid \Lambda}$ ordered according to a fixed orientation of the simple curve $\Lambda\subseteq \Bbbk$. 

The proof of \cite[Theorem 2.1]{2501.06840v2}, analogous to \Cref{th:gen.herm}, relies on leveraging a continuous CS preserver to induce a continuous self-map $\Phi=\Phi_{\phi}$ of the \emph{Grassmannian} \cite[\S 3.3.2]{3264}
\begin{equation*}
  \bG
  =
  \bG(\bC^n)
  :=
  \bigsqcup_{1\le d\le n-1}\bG\left(d,\bC^n\right)
  ,\quad
  \bG(d,\cV):=\left\{\text{$d$-dimensional subspaces of }\cV\right\},
\end{equation*}
appropriately compatible with the \emph{lattice} \cite[Deﬁnition O-1.8]{ghklms_cont-latt_2003} operations of taking space sums $\vee$ and intersections $\wedge$, at which point the \emph{Fundamental Theorem of projective geometry} \cite[Theorem 3.1]{zbMATH01747827} becomes applicable. We will see in due course that the present setup affords some of the same machinery. In preparation for that:

\begin{notation}\label{not:multset.eigspcs}
  Let $T\in \cM_n(\bC)$ be a semisimple operator.
  \begin{enumerate}[(1),wide]
  \item\label{item:not:multset.eigspcs:kl} For $\Lambda\subseteq \bC$ we write
    \begin{equation}\label{eq:klt}
      \cK_{\Lambda}(T):=\sum_{\lambda\in \Lambda}\ker\left(\lambda-T\right).
    \end{equation}

  \item\label{item:not:multset.eigspcs:enum} If the spectrum $\sigma(T)$ is contained in an oriented simple curve $\Lambda\subseteq \bC$, so that \Cref{def:orient}'s order $\le$ is in scope, we enumerate the eigenvalues as $\lambda_1(T)\le \cdots\le \lambda_n(T)$.
    
  \item\label{item:not:multset.eigspcs:ks} More generally, for $S\subseteq [n]:=\left\{1..n\right\}$ set
    \begin{equation*}
      \lambda_S(T):=\left\{\lambda_i(T)\ :\ i\in S\right\}
      ,\quad
      \cK_S(T):=\cK_{\lambda_S(T)}(T)
      \quad
      \left(\text{the latter as in \Cref{eq:klt}}\right).
    \end{equation*}
    Note that $\cK_S(T)+\cK_{[n]\setminus S}(T)=\bC^n$, with the sum direct when $\lambda_S$ and $\lambda_{[n]\setminus S}$ happen not to overlap (e.g. for simple operators).

  \item\label{item:not:multset.eigspcs:tblx} Building on \Cref{item:not:multset.eigspcs:ks}, consider a \emph{tableau $\crc{\mu}$ of shape}
    \begin{equation*}
      \left(\mu_1\ge \cdots\ge \mu_s>0\right)
      =
      \mu\vdash n,
    \end{equation*}
    i.e. \cite[\S 4.1]{fh_rep-th} a filling of (the $n$ boxes constituting) the Young diagram $\mu$ with the elements of $[n]$; in symbols, $\crc{\mu}\vdash n$ again. Writing $\crc{\mu}_i\subseteq [n]$ for the set of symbols arrayed along the $i^{th}$ row of $\mu$, define
    \begin{equation*}
      \cK_{\crc{\mu}}(T)
      :=
      \left(\cK_{\crc{\mu}_1}(T),\ \cdots,\ \cK_{\crc{\mu}_s}(T)\right).
    \end{equation*}
    
  \item\label{item:not:multset.eigspcs:spec.dec} The spectral decomposition of $T$ will be denoted by
    \begin{equation*}
      T
      =
      \sum_{\lambda\in \sigma(T)}\lambda E_{\lambda}(T)
      ,\quad
      E_{\lambda}(T)=0
      \text{ for }
      \lambda\not\in \sigma(T)
    \end{equation*}
    with set-subscript variants 
    \begin{equation*}
      E_{\Lambda}(T):=\sum_{\lambda\in \Lambda} E_{\lambda}(T)
      ,\quad
      E_S(T):=E_{\lambda_S(T)}(T)
    \end{equation*}
    (the latter for curve-bound spectra). 
  \end{enumerate}
\end{notation}
The maps \Cref{eq:h2f} can now simply be denoted by either $\left(\cK_i\right)_i$ (occasionally $(\cK_i(-))_i$ for clarity) or $\cK_{\crc{\mu}}$ for the \emph{standard} \cite[Notation]{fult_y_1997} tableau $\crc{\mu}:=\left((1)(2)\cdots(n)\right)$ of shape $\mu:=(11\cdots 1)$.

\begin{proposition}\label{pr:f2f}
  For an oriented simple curve $\Lambda\subseteq \Bbbk$ a continuous CS preserver $\cH_{n\mid \Lambda}(\Bbbk)\xrightarrow{\phi}\cM_n(\bC)$, $n\in \bZ_{\ge 1}$ produces a commutative diagram
  \begin{equation}\label{eq:f2f}
    \begin{tikzpicture}[>=stealth,auto,baseline=(current  bounding  box.center)]
      \path[anchor=base] 
      (0,0) node (l) {$\cH^{\bullet +}_{n\mid \Lambda}(\Bbbk)$}
      +(3,.5) node (u) {$\bF_{\mu}^{\bullet}(\Bbbk^n)$}

      +(6,0) node (r) {$\bF_{\mu}^{\bullet}(\bC^n)$}
      ;
      \draw[->] (l) to[bend left=6] node[pos=.5,auto] {$\scriptstyle \cK_{\crc{\mu}}$} (u);
      \draw[->,dashed] (u) to[bend left=6] node[pos=.5,auto] {$\scriptstyle \wh{\phi}_{\crc{\mu}}$} (r);
      \draw[->] (l) to[bend right=6] node[pos=.5,auto,swap] {$\scriptstyle \cK_{\crc{\mu}}\circ\phi$} (r);
    \end{tikzpicture}
  \end{equation}
  of continuous maps for $\bullet\in \left\{ss,*\right\}$ for every partition
  \begin{equation*}
    \left(\mu_1\cdots \mu_s\right)
    =
    \mu
    \vdash n
    \quad\text{with}\quad
    \crc{\mu}
    :=
    \left(
      (1\cdots \mu_1)
      \
      (\mu_1+1\cdots \mu_1+\mu_2)
      \ 
      \cdots\right).
  \end{equation*}
\end{proposition}
\begin{proof}
  In other words, the claim is that there is a well-defined dashed arrow factoring the bottom map as depicted. Continuity is again not an issue once $\wh{\phi}_{\crc{\mu}}$ has been defined, so it is the latter claim that is crucial.

  \begin{enumerate}[(I),wide]
  \item\label{item:pr:f2f:pf.clmn}\textbf{: $\mu=(11\cdots 1)$.} Observe first that the fibers $\cK_{\crc{\mu}}^{-1}(\bullet)$ are commuting families of simple operators. CS preservation ensures that every restriction $\left(\cK_{\crc{\mu}}\circ\phi\right)\bigg|_{\cK_{\crc{\mu}}^{-1}(\bullet)}$ is locally constant on the respective fiber, so must be constant by fiber connectedness.

  \item\label{item:pr:f2f:pf.gen}\textbf{: general case.} The difference to the preceding portion of the proof lies in the fibers $\cK_{\crc{\mu}}^{-1}(\bullet)$ no longer being commutative, in general. Having fixed $T,T'$ in a common fiber $\cK_{\crc{\mu}}^{-1}(x)$, $x\in \bF_{\mu}$, continuously deform the eigenvalues of both $T$ and $T'$ so that
    \begin{equation*}
      \lim\left(T|_{\cK_{\crc{\mu}_j}}\right)
      =
      \lambda_j \id|_{\cK_{\crc{\mu}_j}}
      =
      \lim\left(T'|_{\cK_{\crc{\mu}_j}}\right)
      ,\quad
      \forall j
    \end{equation*}
    for fixed $\lambda_j$ (this is possible, as the subsets $\crc{\mu}_j\subseteq [n]$ are contiguous). This has the effect of
    \begin{itemize}[wide]
    \item on the one hand, deforming $\cK_{((1)\cdots(n))}(\phi T,\phi T')$ continuously onto a common $\cK_{\crc{\mu}}\left(\lim T=\lim T'\right)$;

    \item while at the same time keeping those line tuples fixed, by \Cref{item:pr:f2f:pf.clmn} above. 
    \end{itemize}
    The conclusion that $\cK_{\crc{\mu}}(T,T')$ coincide follows. 
  \end{enumerate}
\end{proof}

We will accord some attention in the sequel to the issue of how and to what extent the $\widehat{\phi}$ of \Cref{pr:f2f} fail to be $S_n$-equivariant. In the sequel, the \emph{simple transpositions} in $S_n$ are those of the form $(j\ j+1)$. 

\begin{proposition}\label{pr:eqvr.triv}
  Assume the hypotheses of \Cref{pr:f2f}, and set $\wh{\phi}:=\wh{\phi}_{\left((1)\cdots(n)\right)}$.
  \begin{enumerate}[(1),wide]
  \item\label{item:pr:eqvr.triv:almost.eqvr} There is a self-map $(-)^{\circ}$ of $S_n$ so that $\wh{\phi}\circ\theta=\theta^{\circ}\circ \wh{\phi}$ for all $\theta\in S_n$. 

  \item\label{item:pr:eqvr.triv:trnsp} Furthermore, we have
    \begin{equation*}
      \forall\left(\text{simple transposition }\tau\in S_n\right)
      \left(\wh{\phi}\circ \tau\in \left\{\tau\circ \wh{\phi},\ \wh{\phi}\right\}\right).
    \end{equation*}
  \end{enumerate}
\end{proposition}
\begin{proof}
  \begin{enumerate}[label={},wide]    
  \item\textbf{\Cref{item:pr:eqvr.triv:almost.eqvr}} There is a left $S_n$-action on $\cH^{\bullet +}_{n\mid \Lambda}(\Bbbk)$ obtained by permuting eigenspaces, rendering the upper left-hand map of \Cref{eq:f2f} $S_n$-equivariant:
    \begin{equation*}
      \sum_{i=1}^n \lambda_i(T)E_i(T)
      =:
      T
      \xmapsto{\quad\theta\in S_n\quad}
      \theta T
      :=
      \sum_{i=1}^n \lambda_i(T)E_{\theta^{-1} i}(T).
    \end{equation*}
    For any $\theta$ and $T\in \cH^{\bullet +}_{n\mid \Lambda}(\Bbbk)$ the eigenvalues and eigenspaces of $\phi T$ and $\phi \theta T$ coincide by CS preservation, so that
    \begin{equation*}
      \forall\left(\theta\in S_n\right)
      \forall\left(T\in \cH^{\bullet +}_{n\mid \Lambda}(\Bbbk)\right)
      \exists\left(\theta^{\circ}(T)\in S_n\right)
      \left(
        \wh{\phi}\theta\left(\cK_i(T)\right)_i
        =
        \theta^{\circ}(T)\wh{\phi}\left(\cK_i(T)\right)_i
      \right).
    \end{equation*}
    This pushes through the upper left-hand map of \Cref{eq:f2f} to 
    \begin{equation*}
      \forall\left(\theta\in S_n\right)
      \forall\left(x\in \bF^{\bullet}(\Bbbk^n)\right)
      \exists\left(\theta^{\circ}(x)\in S_n\right)
      \left(
        \wh{\phi}\theta x
        =
        \theta^{\circ}(x)\wh{\phi}x
      \right).
    \end{equation*}
    The $x$-independence of $\theta^{\circ}(x)$ follows from the connectedness of $\bF^{\bullet}(\Bbbk^n)$ (\Cref{res:fss}\Cref{item:res:fss:fconn}), hence the conclusion.
    
  \item\textbf{\Cref{item:pr:eqvr.triv:trnsp}} Consider a simple transposition $\tau:=(j\ j+1)$, $1\le j\le n-1$. That the eigenspaces of $\phi T$ and $\phi \tau T$ can only differ in indices $j$ and $j+1$ (where they possibly may be interchanged) follows from the fact that $T$ and $\tau T$ can be connected by a continuous path leaving $\lambda_i(T)$, $i\not\in \{j,j+1\}$ in place and continuously deforming the pair $\left(\lambda_{j},\lambda_{j+1}\right)$ into its opposite $\left(\lambda_{j+1},\lambda_{j}\right)$.
  \end{enumerate}
\end{proof}

As a consequence of \Cref{pr:eqvr.triv}, we can address the issue raised above of $S_n$-equivariance in \Cref{eq:f2f}.

\begin{corollary}\label{cor:id.triv}
  Under the hypotheses of \Cref{pr:f2f} the map $\wh{\phi}:=\wh{\phi}_{\left((1)\cdots(n)\right)}$ either is $S_n$-equivariant or factors through $\bF^{\bullet}(\Bbbk^n)/S_n$.
\end{corollary}
\begin{proof}
  The $S_n$-actions on the (co)domain of $\wh{\phi}$ being free, the map $(-)^{\circ}$ of \Cref{pr:eqvr.triv}\Cref{item:pr:eqvr.triv:almost.eqvr} is necessarily an endomorphism of $S_n$. Since morphisms defined on $S_n$ are trivial as soon as they annihilate at least one transposition, \Cref{pr:eqvr.triv}\Cref{item:pr:eqvr.triv:trnsp} implies that $(-)^{\circ}$, if non-trivial, must be the identity. The two options $(-)^{\circ}\in \left\{\id,1\right\}$ precisely correspond to the two possibilities listed in the present statement. 
\end{proof}

The dichotomy of \Cref{cor:id.triv} is suggestive of \Cref{th:gen.herm}'s; \Cref{pr:n3.if.ct} confirms that intuition, handling one branch. 

\begin{proposition}\label{pr:n3.if.ct}
  Let $n\in \bZ_{\ge 3}$ and $\cH_{n\mid \Lambda}\xrightarrow{\phi}\cM_n(\bC)$ a continuous CS preserver for a simple curve $\Lambda\subseteq \Bbbk\in \left\{\bR,\bC\right\}$.

  If the map $\wh{\phi}:=\wh{\phi}_{((1)\cdots(n))}$ factors through $\bF^{ss,*}(\Bbbk^n)/S_n$ then it is constant. 
\end{proposition}

Pausing first for the immediate consequence:

\begin{corollary}\label{cor:n3.if.ct}
  In the context of \Cref{th:gen.herm}, if $\wh{\phi}_{((1)\cdots(n))}$ factors through $\bF^{ss,*}(\Bbbk^n)/S_n$ then $\phi$ is of type \Cref{item:th:gen.herm:fix}.  \qedhere  
\end{corollary}

\pf{pr:n3.if.ct}
\begin{pr:n3.if.ct}
  Consider $x=\left(\ell_i\right)_{i=1}^n\in \bF:=\bF^{ss,*}(\Bbbk^n)$, as well as a perturbation
  \begin{equation*}
    x':=\left(\ell'_i\right)_{i=1}^n
    \in \bF
    ,\quad
    \forall\left(j\ne 1\right)\left(\ell'_j=\ell_j\right)
    ,\quad
    \ell'_1+\ell'_2=\ell_1+\ell_2
  \end{equation*}
  thereof. Operators
  \begin{equation*}
    T,T'\in \cH^{+}
    ,\quad
    \cK_{(1)\cdots(n)}(T,T')
    \xlongequal{\quad\text{respectively}\quad}
    x,x'
  \end{equation*}
  admit continuous deformations
  \begin{equation*}
    T_t
    \xrightarrow[\quad [0,1)\ni t\to 1\quad]{\quad\quad}
    S
    \xleftarrow[\quad 1\leftarrow t\in [0,1)\quad]{\quad\quad}
    T'_t
    ,\quad
    \left[
      \begin{gathered}
        T^{\bullet}=T^{\bullet}_0\\
        \cK_{((1)\cdots(n))}(T_t^{\bullet})=x^{\bullet}\\
        \cK_{((12)(3)\cdots(n))}(S)=\left(\ell_1+\ell_2,\ell_3\cdots \ell_n\right)
      \end{gathered}
    \right.
  \end{equation*}
  altering only the $\lambda_{1,2}$ eigenvalues. Said deformations will keep $\wh{\phi}(x,x')$ constant, meaning that
  \begin{equation*}
    \forall\left(j\ge 3\right)\left(\wh{\phi}(x)_j=\wh{\phi}(x')_{j}\right)
    \quad\text{and}\quad
    \wh{\phi}(x)_1+\wh{\phi}(x)_2
    =
    \wh{\phi}(x')_1+\wh{\phi}(x')_2.
  \end{equation*}
  By the assumed $S_n$-invariance and the assumption that $n\ge 3$, however, this also gives $\wh{\phi}(x)_j=\wh{\phi}(x')_j$ for \emph{all} $j\in [n]$. The conclusion follows by noting that any two $x,x'\in \bF$ can be connected by a chain
  \begin{equation*}
    x
    =
    x_0
    ,\ x_1
    ,\ \cdots
    ,\ x_s
    =
    x'
  \end{equation*}
  with consecutive $x_{p,p+1}$
  \begin{itemize}[wide]
  \item differing in only one component $x_{p,j}\ne x_{p+1,j}$ for some $j\in [n]$;
  \item so that the 2-planes $x_{p,j}+x_{p,j'}$ and $x_{p+1,j}+x_{p+1,j'}$ coincide for some $j'\ne j$.  \qedhere
  \end{itemize}
\end{pr:n3.if.ct}

\Cref{cor:n3.if.ct} turns the focus on the yet-to-be-examined option in \Cref{cor:id.triv}.

\begin{lemma}\label{le:max.ab}
  Let $n\in \bZ_{\ge 1}$ and $\cH_{n\mid \Lambda}\xrightarrow{\phi}\cM_n(\bC)$ a continuous CS preserver for a simple curve $\Lambda\subseteq \Bbbk\in \left\{\bR,\bC\right\}$.

  If the map $\wh{\phi}:=\wh{\phi}_{((1)\cdots(n))}$ is $S_n$-equivariant then $\phi$ restricts to a conjugation on every maximal abelian subset of $\cH_{n\mid\Lambda}$. 
\end{lemma}
\begin{proof}
  Said maximal abelian subsets are precisely the
  \begin{equation*}
    \overline{
      \bigsqcup_{\substack{\crc{\mu}\\\mu:=(11\cdots 1)}}
      \cK_{\crc{\mu}}^{-1}(x)
    }
    \subseteq
    \cH_{n\mid \Lambda}
    ,\quad
    x\in \bF^{\bullet}(\Bbbk^n)
  \end{equation*}
  ($\overline{(-)}$ denoting closure), and the conclusion is immediate from CS preservation (which delivers that conclusion for a single tableau $\crc{\mu}$) coupled with the assumed $S_n$-equivariance (which ensures compatibility among the $n!$ tableaux $\crc{\mu}$).
\end{proof}

\pf{th:gen.herm}
\begin{th:gen.herm}
  The domain of $\phi$ consists of normal operators, consequently with mutually-orthogonal eigenspaces. The map $\wh{\phi}:=\wh{\phi}_{((1)\cdots(n))}$ introduced in \Cref{pr:f2f} is henceforth assumed $S_n$-equivariant, as afforded by \Cref{cor:n3.if.ct}. That result having disposed of the \Cref{item:th:gen.herm:fix} side of the present theorem, the goal is to argue that $\phi$ is of type \Cref{item:th:gen.herm:cj}.
  
  We will construct a continuous, dimension- and inclusion-preserving map $\bG(\Bbbk^n)\xrightarrow{\Phi=\Phi_{\phi}}\bG(\bC^n)$ that recovers $\wh{\phi}$ in the sense that 
  \begin{equation*}%\label{eq:phi.phi}
    \wh{\phi}\left(\ell_1,\ \cdots,\ \ell_n\right)
    =
    \left(\Phi\ell_1,\ \cdots,\ \Phi\ell_n\right).
  \end{equation*}
  $\Phi$ will furthermore respect lattice operations for pairs of spaces whose respective orthogonal projections commute; or: writing
  \begin{equation*}
    \cV\operp \cW
    \quad
    \text{for}
    \quad
    \left(\cV\cap \cW\right)
    \perp
    \left(\cV\ominus \left(\cV\cap \cW\right)\right)
    ,\
    \left(\cW\ominus \left(\cV\cap \cW\right)\right)
  \end{equation*}
  (with `$\ominus$' denoting the orthogonal complement of its right-hand side in the left),
  \begin{equation}\label{eq:if.operp}
    \cV
    \operp
    \cW
    \quad
    \xRightarrow{\quad}
    \quad
    \Phi\left(\cV\vee\cW\right)=\Phi\cV\vee\Phi\cW
    \quad\text{and}\quad
    \Phi\left(\cV\wedge\cW\right)=\Phi\cV\wedge\Phi\cW.
  \end{equation}
  
  We begin by defining the individual components $\Phi_d:=\Phi|_{\bG(d,\Bbbk^n)}$, $d\in [n]$ by
  \begin{equation}\label{eq:phid.def}
    \bG(d,\Bbbk^n)
    \ni
    \cK_{[d]}(T)
    \xmapsto[\quad T\in \cH^{*+}_{n\mid \Lambda}\quad]{\quad\Phi_d\quad}
    \cK_{[d]}(\phi T)
    \in
    \bG(d,\bC^n)
  \end{equation}
  for $\cK_{\bullet}$ as in \Cref{not:multset.eigspcs}\Cref{item:not:multset.eigspcs:ks}. \emph{Were} the definition consistent, continuity, dimension preservation and inclusion preservation would be routine.
  
  \begin{enumerate}[(I),wide]
  \item\textbf{: The $\Phi_d$ are well-defined.} This is a variant of the argument employed in \cite[Lemma 1.10]{2505.19393v3}, say: having fixed $a<b\in \Lambda$, note that all $T$ candidates for \Cref{eq:phid.def} commute with the operator $T_W$ with eigenvalue $a$ along $W:=\cK_{[d]}(T)$ and $b$ along $W^{\perp}$. There are continuous curves
    \begin{equation*}
      \left(T_t\right)_{t\in [0,1]}
      ,\quad
      T_0=T
      ,\quad
      T_1=T_W
      ,\quad
      \forall(t\in[0,1))\left(T_t\in \cH^{*+}\right)
    \end{equation*}
    preserving eigenspaces, whence $\cK_{[d]}(T)=\cK_{[d]}(T_W)$. The latter of course depends on $W$ only, and the consistency of \Cref{eq:phid.def} follows.

  \item\textbf{: \Cref{eq:if.operp} holds.} Observe first that the $S_n$-equivariance of $\wh{\phi}$ allows us to recast \Cref{eq:phid.def} as
    \begin{equation*}
      \bG(d,\Bbbk^n)
      \ni
      \cK_{S}(T)
      \xmapsto[\quad T\in \cH^{*+}_{n\mid \Lambda}\quad]{\quad\Phi_d\quad}
      \cK_{S}(\phi T)
      \in
      \bG(d,\bC^n)
    \end{equation*}
    for \emph{any} $d$-sized $S\subseteq [n]$. \Cref{eq:if.operp} will now follow; for suprema, for instance (i.e. sums; intersections are handled similarly) observe that whenever
    \begin{equation*}
      \cV\operp \cV'
      ,\quad
      d:=\dim \cV
      ,\quad
      d':=\dim \cV'
      ,\quad
      d_0:=\dim \cV\cap \cV'
    \end{equation*}
    we can select subsets $S,S'\subseteq [n]$ of respective cardinalities $d,d'$ with $d_0$-sized intersection $S_0:=S\cap S'$ and
    \begin{equation*}
      T\in \cH^{*+}_{n\mid \Lambda}
      ,\quad
      \cK_S(T)=\cV
      ,\quad
      \cK_{S'}(T)=\cV'
      ,\quad
      \cK_{S_0}(T)=\cV\cap \cV',
    \end{equation*}
    yielding
    \begin{equation*}
      \Psi_{d+d'-d_0}\left(\cV+\cV'=\cK_{S\cup S'}(T)\right)
      =
      \cK_{S\cup S'}(\phi T)
      =
      \cK_{S}(\phi T)
      +
      \cK_{S'}(\phi T)
      =
      \Psi_d \cV
      +
      \Psi_{d'}\cV'.
    \end{equation*}
    
  \item\textbf{: Conclusion.} $\Phi$ in hand, we can proceed as in the proof of \cite[Theorem 2.1, unitary portion]{2501.06840v2} (which strategy the present argument adapts). \cite[Proposition 2.5]{2501.06840v2} ensures\footnote{That result assumes the domain is again the \emph{complex} Grassmannian, but the argument goes through for $\Phi$ defined on $\bG(\bR^n)$ instead.} the existence of a linear or conjugate-linear 
    \begin{equation*}
      \Bbbk^n
      \lhook\joinrel\xrightarrow[\quad]{\quad J\quad}
      \bC^n
      ,\quad
      \forall\left(\cV\in \bG(\Bbbk^n)\right)
      \left(
        \Phi \cV=J\cV
      \right).
    \end{equation*}
    $J$ thus maps $\lambda$-eigenspaces of $T\in \cH^*_{n\mid \Lambda}$ respectively onto $\lambda$-eigenspaces of $\phi T$, so $\phi$ is either
    \begin{itemize}[wide]
    \item conjugation by $J$ if the latter is linear;

    \item or
      \begin{equation*}
        \Ad_{J}(-)^*
        =
        \Ad_{JJ'}()^t
      \end{equation*}
      if $J$ is conjugate-linear, with $J'$ denoting standard complex conjugation on $\bC^n$ (with respect to the basis assumed fixed in denoting that space by `$\bC^n$' to begin with).  \qedhere
    \end{itemize}
  \end{enumerate}
\end{th:gen.herm}

\pf{th:scl}
\begin{th:scl}
  \Cref{th:gen.herm} ensures that every restriction
  \begin{equation*}
    \phi|_{\cH^*_{n\mid \Lambda_p}}
    ,\quad
    p\in \Lambda
    ,\quad
    \Lambda_p:=\Lambda\setminus \{p\}
  \end{equation*}
  is either of type \Cref{item:th:gen.herm:cj} or \Cref{item:th:gen.herm:fix}. Continuity implies type coherence for varying $p\in \Lambda$, so it will be enough to rule out the type-\Cref{item:th:gen.herm:fix} possibility.

  Assume for a contradiction that
  \begin{equation*}
    \forall\left(p\in\Lambda\right)
    \left(
      \cH^*_{n\mid \Lambda_p}
      \ni
      X
      \xmapsto{\quad\phi\quad}
      \mathrm{diag}\left(\lambda_1(X)\cdots\lambda_n(X)\right)
    \right)    
  \end{equation*}
  for the counterclockwise ordering along $\Lambda$ (so also along each $\Lambda_p$). The inconsistency is plain: for a simple operator $T\in \cH^{*+}_{n\mid\Lambda_p}$ the counterclockwise spectrum orderings will differ depending on which consecutive spectrum elements $p$ separates. 
\end{th:scl}

\pf{th:if.l.reg}
\begin{th:if.l.reg}
  \begin{enumerate}[label={},wide]
  \item\textbf{\Cref{item:th:if.l.reg:is.cs}} CS preservers form a monoid under composition, and we have already observed repeatedly that maps of the form \Cref{item:th:gen.herm:cj,item:th:gen.herm:fix} will do. That $(-)^{\rho}$ is well-defined is a consequence of \cite[Lemma 6.2]{zbMATH06285212}; the proof of \cite[Proposition 2.14]{2501.06840v2} argues this as well in passing, along with commutativity preservation as a consequence of the \emph{Putnam-Fuglede theorem} \cite[p.376, second statement]{zbMATH03133061}.

  \item\textbf{\Cref{item:th:if.l.reg:ddb} and \Cref{item:th:if.l.reg:ddc}} Noting that
    \begin{equation*}
      \begin{tikzpicture}[>=stealth,auto,baseline=(current  bounding  box.center)]
        \path[anchor=base] 
        (0,0) node (l) {$\cM^{ss}_{n}$}
        +(2,.5) node (u) {$\cM^{ss}_n$}
        +(4,0) node (r) {$\cM^{ss}_n$}
        ;
        \draw[->] (l) to[bend left=6] node[pos=.5,auto] {$\scriptstyle (-)^{\rho}$} (u);
        \draw[->] (u) to[bend left=6] node[pos=.5,auto] {$\scriptstyle (-)^*$} (r);
        \draw[->] (l) to[bend right=6] node[pos=.5,auto,swap] {$\scriptstyle \Ad_L N\xmapsto{\quad} \Ad_L N^*$} (r);
      \end{tikzpicture}
    \end{equation*}
    is precisely the self-map (of the space $\cM_n^{ss}$ of semisimple operators) applying complex conjugation $\overline{\bullet}$ to every eigenvalue and preserving the respective eigenspaces, the two statements are precisely what \cite[Theorem 4.3 (ii) $\Leftrightarrow$ (iv)]{zbMATH06285212}  and \cite[Proposition 4.5 (i) $\Leftrightarrow$ (iii)]{zbMATH06285212} respectively provide.  \qedhere
  \end{enumerate}
\end{th:if.l.reg}

%%%%%%%%%%%%%%%%%%%%%%%%%%%%%%%%
%%%%%%%%%%%%%%%%%%%%%%%%%%%%%%%%

\addcontentsline{toc}{section}{References}
%\bibliography{bib}{}
%\bibliographystyle{plain}

% BEGIN INSERTED BBL (crv-spc-xv2.bbl)
\def\polhk#1{\setbox0=\hbox{#1}{\ooalign{\hidewidth
  \lower1.5ex\hbox{`}\hidewidth\crcr\unhbox0}}}

% END INSERTED BBL

\Addresses

\end{document}